\documentclass[12pt,reqno]{amsart}
\usepackage{fullpage}
\usepackage{times}
\usepackage{amsmath,amssymb,amsthm,url}
\usepackage[utf8]{inputenc}
\usepackage[english]{babel}
\usepackage{graphicx}
\usepackage[colorlinks=true, pdfstartview=FitH, linkcolor=blue, citecolor=blue, urlcolor=blue]{hyperref}

\numberwithin{equation}{section}
\newtheorem{theorem}{Theorem}[section]
\newtheorem{lemma}[theorem]{Lemma}
\newtheorem{corollary}[theorem]{Corollary}
\newtheorem{proposition}[theorem]{Proposition}

\theoremstyle{definition}
\newtheorem{definition}[theorem]{Definition}
\newtheorem{notation}[theorem]{Notation}
\newtheorem{remark}[theorem]{Remark}

\newcommand{\B}{{\mathcal B}}
\newcommand{\C}{{\mathbb C}}
\newcommand{\cD}{{\mathcal D}}
\newcommand{\cI}{{\mathcal I}}
\newcommand{\N}{{\mathbb N}}
\newcommand{\cN}{{\mathcal N}}
\renewcommand{\P}{{\mathcal P}}
\newcommand{\Q}{{\mathbb Q}}
\newcommand{\cR}{{\mathcal R}}
\newcommand{\T}{{\mathcal T}}
\newcommand{\Z}{{\mathbb Z}}

\newcommand{\ord}{\mathop{\rm ord}}

\renewcommand{\mod}[1]{{\ifmmode\text{\rm\ (mod~$#1$)}\else\discretionary{}{}{\hbox{ }}\rm(mod~$#1$)\fi}}
\newcommand{\ep}{\varepsilon}

\usepackage{scalerel}

\begin{document}

\title{The smallest invariant factor of the multiplicative group}
\author{Ben Chang}
\address{Department of Mathematics \\ University of Toronto \\ Bahen Centre \\ 40 St.~George St., Room 6290 \\ Toronto, Ontario, Canada \ M5S 2E4}
\email{bchang@math.toronto.edu}
\author{Greg Martin}
\address{Department of Mathematics \\ University of British Columbia \\ Room 121, 1984 Mathematics Road \\ Vancouver, BC, Canada \ V6T 1Z2}
\email{gerg@math.ubc.ca}
\maketitle

\begin{abstract}
Let $\lambda_1(n)$ denote the least invariant factor in the invariant factor decomposition of the multiplicative group $M_n = (\Z/n\Z)^\times$. We give an asymptotic formula, with order of magnitude $x/\sqrt{\log x}$, for the counting function of those integers~$n$ for which $\lambda_1(n)\ne2$. We also give an asymptotic formula, for any even $q\ge4$, for the counting function of those integers~$n$ for which $\lambda_1(n)=q$. These results require a version of the Selberg--Delange method whose dependence on certain parameters is made explicit, which we provide in an appendix. As an application, we give an asymptotic formula for the counting function of those integers~$n$ all of whose prime factors lie in an arbitrary fixed set of reduced residue classes, with implicit constants uniform over all moduli and sets of residue classes.
\end{abstract}


\section{Introduction} \label{intro sec}

The multiplicative group $M_n = (\Z/n\Z)^\times$ of the quotient ring $\Z/n\Z$ is a finite abelian group whose structure is closely tied to functions of interest to number theorists. Most obviously, its cardinality is the Euler function $\phi(n)$, whose extreme values and distribution have been extensively studied. However, there can be many abelian groups of cardinality $\phi(n)$, which leads us to wonder which such group $M_n$ is; for example, asking whether $M_n$ is cyclic is precisely the same question as characterizing those moduli~$n$ that possess primitive roots.

To study the structure of this family of groups more finely, we may ask about the invariant factor decomposition of~$M_n$, namely the unique way of writing $M_n \cong \Z/d_1\Z \times \Z/d_2\Z \times \dots \Z/d_\ell\Z$ where $d_1\mid d_2\mid \cdots \mid d_\ell$. For example, it is not hard to show that the length $\ell$ of the invariant factor decomposition of~$M_n$ is essentially the number of distinct prime factors of~$n$ (more precisely, these two quantities differ by $1$, $0$, or $-1$ depending upon the power of~$2$ dividing~$n$). Moreover, the largest invariant factor $d_\ell$ is the exponent of the group $M_n$, or precisely the Carmichael function $\lambda(n)$ which has also been well studied (see~\cite{EPS}).

In this article, we investigate instead the distribution of the smallest invariant factor~$d_1$ of the multiplicative group~$M_n$, which we denote by $\lambda_1(n)$. Other than for $n=1$ and $n=2$, this smallest invariant factor is always even (see the proof of Proposition~\ref{lif to factorization prop} below and the subsequent remark). It turns out that $\lambda_1(n)=2$ for almost all integers~$n$; however, we can be more precise about the exceptions, as our first theorem asserts.

\begin{definition} \label{Tx def}
Define $T(x)$ to be the number of positive integers $n\le x$ such that the least invariant factor~$\lambda_1(n)$ of~$M_n$ does not equal~$2$.
\end{definition}

We prove:

\begin{theorem} \label{T theorem}
For $x\ge 2$, we have $T(x) = (H_4+H_6)\dfrac{x}{\sqrt{\log x}} + O \bigg( \dfrac x{(\log x)^{3/4}} \bigg)$, where
\begin{equation} \label{H4 H6 def}
\begin{split}
H_4 &= \frac3{2^{5/2}} \prod_{p\equiv3\mod4} \bigg( 1-\frac1{p^2} \bigg)^{1/2} \approx 0.490694 \\
H_6 &= \frac7{2^{5/2} \cdot 3^{3/4}} \prod_{p\equiv5\mod6} \bigg( 1-\frac1{p^2} \bigg)^{1/2} \approx 0.527129.
\end{split}
\end{equation}
\end{theorem}

Indeed, our methods allow us to establish an asymptotic formula, for any fixed even number $q\ge4$, for the counting function of those integers~$n$ for which $\lambda_1(n)=q$.

\begin{definition} \label{Eq def}
Define $E_q(x) = \#\{n\le x\colon \lambda_1(n)=q\}$ to be the number of positive integers $n\le x$ such that the least invariant factor of $n$ equals~$q$.
\end{definition}

\begin{theorem} \label{Eq theorem}
Let $q\ge4$ be an even integer. For $x\ge e^e$,
\begin{equation*}
E_q(x) = \frac{H_qx}{(\log x)^{1-1/\phi(q)}} + O_q \bigg( \frac{x(\log\log x)^2}{(\log x)^{1-1/2\phi(q)}} \bigg).
\end{equation*}
Here the constant $H_q$ is defined by
\begin{equation} \label{Hq def}
H_q = \frac{3 G_q}2 \cdot \begin{cases}
1 + \dfrac1{p^r(p-1)}, &\text{if there exists an odd prime power $p^r\|q$ such that } \dfrac q{p^r} \mid (p-1), \\
1, &\text{otherwise,}
\end{cases}
\end{equation}
where
\begin{equation} \label{Gq def}
G_q = \frac1{\Gamma(1/\phi(q))} \bigg( \frac{\phi(q)}q \prod_{\substack{\chi\mod q \\ \chi \ne \chi_0}} L(1,\chi) \bigg)^{1/\phi(q)} \prod_{\substack{p \nmid q \\ p \not\equiv 1 \mod{q}}} \bigg(1-\frac{1}{p^{\ord_q(p)}} \bigg)^{1/\ord_q(p)}.
\end{equation}
\end{theorem}

\noindent In equation~\eqref{Gq def} and hereafter, we use $\ord_q(p)$ to denote the order of $p$ modulo~$q$. We show, in the proof of Proposition~\ref{Hq prop} below, that the odd prime power $p^r$ (including the possibility $r=1$) referred to in equation~\eqref{Hq def} is unique if it exists.

For example, if $\chi_{-4}$ and $\chi_{-6}$ denote the nonprincipal characters modulo~$4$ and~$6$, respectively, then $H_4 = \frac32 G_4$ where
\[
G_4 = \frac1{\Gamma(1/2)} \bigg( \frac24 L(1,\chi_{-4}) \bigg)^{1/2} \prod_{p\equiv3\mod4} \bigg( 1-\frac1{p^2} \bigg)^{1/2} = \frac1{\sqrt\pi} \bigg( \frac\pi8 \bigg)^{1/2} \prod_{p\equiv3\mod4} \bigg( 1-\frac1{p^2} \bigg)^{1/2},
\]
while $H_6 = \frac32 G_6\cdot\frac76$ where
\[
G_6 = \frac1{\Gamma(1/2)} \bigg( \frac26 L(1,\chi_{-6}) \bigg)^{1/2} \prod_{p\equiv5\mod6} \bigg( 1-\frac1{p^2} \bigg)^{1/2} = \frac1{\sqrt\pi} \bigg( \frac\pi{6\sqrt3} \bigg)^{1/2} \prod_{p\equiv5\mod6} \bigg( 1-\frac1{p^2} \bigg)^{1/2}.
\]
Both evaluations match the expressions given in equation~\eqref{H4 H6 def}, which is not coincidental: the contribution to the main term for $T(x)$ in Theorem~\ref{T theorem} is actually from those integers~$n$ for which $\lambda_1(n)=4$ or $\lambda_1(n)=6$. Indeed, Theorem~\ref{Eq theorem} shows that any individual $q\ge8$ leads to a counting function $E_q(x)$ with a smaller order of magnitude than either $E_4(x)$ or $E_6(x)$. We cannot directly deduce Theorem~\ref{T theorem} from Theorem~\ref{Eq theorem}, however, because of uniformity issues with the asymptotic formula in the latter theorem.

In fact, uniformity issues present the main technical obstacle to proving Theorem~\ref{Eq theorem}, for a different reason:
instead of those integers~$n$ for which the smallest invariant factor is exactly~$q$,
it is more approachable to first count those integers~$n$ for which the smallest invariant factor is a multiple of~$q$.

\begin{definition} \label{Dq def}
Define $D_q(x) = \#\{n\le x\colon q\mid\lambda_1(n)\}$ to be the number of positive integers $n\le x$ such that $q$ divides the least invariant factor of~$n$.
\end{definition}

\noindent As it turns out, the asymptotic formula for $D_q(x)$ will be the same as the asymptotic formula in Theorem~\ref{Eq theorem} for $E_q(x)$ (see Proposition~\ref{Hq prop} below). We can then recover the functions $E_q(x)$ from the functions $D_q(x)$ via inclusion--exclusion; however, this recovery step requires some quantitative statement of uniformity in~$q$.

We will see in Proposition~\ref{lif to factorization prop} that, roughly speaking, an integer is counted by $D_q(x)$ when all of its prime factors are congruent to 1\mod q. Counting integers of this form is a standard goal in analytic number theory, going all the way back to Landau, and can be effectively done using the Selberg--Delange method. An excellent account of the Selberg--Delange method appears in the book of Tenenbaum~\cite[Chapter~II.5]{tenenbaum15}. However, that version allows implicit constants to depend on three parameters~$\delta$, $c_0$, and~$A$; for our purposes, we require a version where the dependence upon the parameters~$\delta$ and~$c_0$ is explicit. Consequently, we establish such a uniform version of the Selberg--Delange method (closely following Tenenbaum's proof for the most part) in Theorem~\ref{thm_SD}; this additional uniformity might be of value in other research as well, so we have retained the generality already present.

The techniques that establish the above results could also be used to investigate the other canonical representation of the abelian group~$M_n$, namely its primary decomposition as a direct sum of cyclic groups of prime-power order. For example, one can see that the least such primary factor of~$M_n$ equals~$2$ unless $4\nmid n$ and every odd prime dividing~$n$ is congruent to $1\mod4$. Therefore the counting function of those integers whose least primary factor does not equal $2$ is precisely $D_4(x)$ (see Proposition~\ref{lif to factorization prop} below), which is asymptotic to $E_4(x)$ and thus to $H_4x/\sqrt{\log x}$ where $H_4$ was defined in equation~\eqref{H4 H6 def}. The methods of this paper would provide an asymptotic formula for the counting function of integers with any prescribed least primary factor, after carefully characterizing such integers $n$ in terms of their prime factorizations; but we do not carry out the details of that variant herein.

After the preliminary Section~\ref{complex section}, we give in Section~\ref{arbitrary residues sec} an asymptotic formula for the counting function of integers whose prime factors are restricted to lie in an arbitrary fixed set of reduced residue classes (see Theorem~\ref{general B thm}), with explicit dependence on the modulus and the number of residue classes; we hope that such a theorem might find further use elsewhere. We specialize in Section~\ref{special 1} to the case of counting integers all of whose prime factors are congruent to 1\mod q; asymptotic formulas for this function can be found in the literature, but not with the uniformity we require in this application. Section~\ref{algebra to number theory sec} contains the algebra and elementary number theory necessary to give the precise relationship between this counting function and $D_q(x)$, the asymptotic formula for which is provided in Proposition~\ref{Hq prop}. From this result, we proceed to establish Theorem~\ref{Eq theorem} in Section~\ref{second thm sec} and finally Theorem~\ref{T theorem} in Section~\ref{main thm sec}. The appendix is devoted to establishing the uniform version of the Selberg--Delange method described above.

\section{Classical complex analytic results} \label{complex section}

We will employ the classical ``zero-free regions'' for Dirichlet $L$-functions, with a slight wrinkle related to the fact that we need such a region to be free of even a possible exceptional zero. Using the notational convention that $s=\sigma+i\tau$, we first state the standard result (see~\cite[Corollary~11.8 and Theorem~11.4]{MV}, although note that the $\tau$ therein is $|\tau|+4$ in our notation):

\begin{lemma} \label{nonexceptional bound}
There exists an effective constant $c>0$ such that for all $q\ge3$ and all Dirichlet characters $\chi\mod q$, the function $L(s,\chi)$ has no zeros in the region
$\sigma \ge 1-c/\log(q(|\tau|+4))$,
except possibly for a single exceptional real zero of a single character $\chi\mod q$. In the smaller region
\begin{equation} \label{preliminary region}
\sigma \ge 1-\frac{c/2}{\log(q(|\tau|+4))},
\end{equation}
we have
\begin{equation} \label{nonexceptional ineq}
|\log L(s,\chi)| \le \log\log(q(|\tau|+4)) + O(1)
\end{equation}
for every nonexceptional character $\chi\mod q$.
\end{lemma}

On the other hand, we have sufficient control over exceptional zeros and nearby values of $L$-functions (see~\cite[Corollary~11.12 and equations~(11.9) and~(11.10)]{MV} and note that $|\log L(s,\chi)| \le \big| \log |L(s,\chi)| \big| + |\arg L(s,\chi)|$):

\begin{lemma} \label{exceptional bound}
There exists an effective constant $c>0$ such that any exceptional zero~$\beta$ for $L(s,\chi)$ with $\chi\mod q$ is at most $1-c/q^{1/2}\log^2q$. In the region~\eqref{preliminary region}, if $|s-\beta| > 1/\log q$ then $L(s,\chi)$ satisfies the bound~\eqref{nonexceptional ineq}; on the other hand, if $|s-\beta| \le 1/\log q$ then
\[
|\log L(s,\chi)| \le \log|s-\beta|^{-1} + \log\log q+O(1).
\]
\end{lemma}

We combine these two lemmas into the following easy-to-cite proposition:

\begin{proposition} \label{zerofree prop}
There exists an effective constant $0<\eta<\frac1{81}$ such that for all $q\ge3$ and all Dirichlet characters $\chi\mod q$, the function $L(s,\chi)$ has no zeros in the region
\begin{equation} \label{region of a.c.}
\sigma \ge 1 - \frac{c_0}{1+\log^+ |\tau|} \quad\text{with } c_0 = \frac\eta{q^{1/2}\log^2q}
\end{equation}
and satisfies the bound
\begin{equation} \label{combined bound}
|\log L(s,\chi)| \le \begin{cases}
\log\log^+(q\tau) + O(1), &\text{if $\chi$ has no exceptional zero}, \\
\tfrac12\log q + 3\log\log^+(q\tau) + O(1), &\text{if $\chi$ has an exceptional zero $\beta$}
\end{cases}
\end{equation}
inside this region.
\end{proposition}

\begin{proof}
The derivation needs no comment other than the remark that we take $\eta<c/2$ in these lemmas, so that in particular
\[
\log|s-\beta|^{-1} \le \log \frac{q^{1/2}\log^2q}\eta = \tfrac12\log q + 2\log\log q + O(1).
\qedhere
\]
\end{proof}

We next take a moment to establish a convenient upper bound for complex logarithms and exponential functions defined by continuous variation on a domain. For fixed $z\in\C$, our convention will be to choose the branch of $w^z$ (on regions to be specified) so that $1^z=1$. Often we will set $w$ equal to some Dirichlet series~$F(s)$ with constant term~$1$, so that we choose the branch of $F(s)^z$ that tends to~$1$ as $\Re s\to\infty$. It would be easy to write down a false inequality: it is not the case, for example, that $|w^z| \le |w|^{|z|}$, as the example $i^{-i} = e^{\pi/2}$ shows (or even the example $0.1^{-2}$). But we can use a similar-looking inequality:

\begin{lemma} \label{w^z lemma}
Let $w\ne0$ and $z$ be complex numbers, and fix any value of $\log w$.
If $|\log w| \le L$, then $|w^z| \le e^{L|z|}$.
\end{lemma}

\begin{proof}
We have $|w^z| = |e^{z\log w}| = e^{\Re(z\log w)} \le e^{|z\log w|} \le e^{|z|L}$.
\end{proof}

\section{Integers with prime factors from prescribed residue classes} \label{arbitrary residues sec}

The following notation will be used throughout this section.

\begin{notation} \label{B beta notation}
Let $q\ge3$ be an integer. Let $\B$ be the union of $B$ distinct reduced residue classes\mod q. Set $\beta=B/\phi(q)$ and $\underline B = \min\{B,\phi(q)-B\}$. Let $\cN_\B = \{n\in\N\colon p\mid n\Rightarrow p\in\B\}$ be the set of all positive integers whose prime factors lie exclusively in~$\B$, and define the associated Dirichlet series
\[
F_\B(s) = \sum_{n\in\cN_\B} \frac1{n^s} = \prod_{p\in\B} \bigg( 1-\frac1{p^s} \bigg)^{-1},
\]
which converges absolutely when $\sigma>1$. Finally, let $\zeta(s)^{-\beta}$ denote the branch of the complex exponential chosen, as discussed earlier, so that $\lim_{\Re s\to\infty} \zeta(s)^{-\beta}=1$, and set $G_\B(s) = F_\B(s) \zeta(s)^{-\beta}$, which {\em a priori} is defined only when $\sigma>1$.
\end{notation}

\begin{lemma} \label{property 1 lemma}
The Dirichlet series $G_\B(s)$ has an analytic continuation to the region~\eqref{region of a.c.}.
\end{lemma}

\noindent Those familiar with the Selberg--Delange method, or those who refer to Definition~\ref{property P def}, might recognize that we are beginning to establish that $F(s)$ has the property $\P(\beta;c_0,\frac12,M)$, with $\beta$ as in Notation~\ref{B beta notation} and with $c_0$ as in equation~\eqref{region of a.c.}.

\begin{proof}
For $\sigma>1$, define
\begin{equation} \label{AB def}
A_\B(s) = F_\B(s)^{\phi(q)} \prod_{\chi\mod q} L(s,\chi)^{-\sum_{b\in\B} \overline\chi(b)}.
\end{equation}
For any prime~$p$, note that
\begin{align*}
\log \prod_{\chi\mod q} \bigg( 1-\frac{\chi(p)}{p^s} \bigg)^{\sum_{b\in\B} \overline\chi(b)} &= \sum_{\chi\mod q} \sum_{b\in\B} \overline\chi(b) \log \bigg( 1-\frac{\chi(p)}{p^s} \bigg) \\
&= \sum_{b\in\B} \sum_{\chi\mod q} \overline\chi(b) \bigg( {-}\frac{\chi(p)}{p^s} + O\bigg( \frac1{p^{2\sigma}} \bigg) \bigg) \\
&= -\frac1{p^s} \begin{cases} \phi(q), &\text{if } b\in\B, \\ 0, &\text{if } b\notin\B \end{cases} \Bigg\} + O\bigg( \frac{\phi(q)^2}{p^{2\sigma}} \bigg)
\end{align*}
by orthogonality, whence
\begin{align}
A_\B(s) &= \prod_{p\in\B} \bigg( 1-\frac1{p^s} \bigg)^{-\phi(q)} \prod_p \Bigg( 1 -\frac1{p^s} \begin{cases} \phi(q), &\text{if } b\in\B, \\ 0, &\text{if } b\notin\B \end{cases} \Bigg\} + O\bigg( \frac{\phi(q)^2}{p^{2\sigma}} \bigg) \Bigg) \notag \\
&= \prod_p \bigg( 1 + O\bigg( \frac{\phi(q)^2}{p^{2\sigma}} \bigg) \bigg). \label{gonna be crude}
\end{align}
In particular, this infinite product has an analytic continuation to $\sigma>\frac12$, in which it converges absolutely and therefore does not vanish. Then
\begin{align}
G_\B(s)^{\phi(q)} = \zeta(s)^{-\beta\phi(q)} F_\B(s)^{\phi(q)} &= \zeta(s)^{-B} A_\B(s) \prod_{\chi\mod q} L(s,\chi)^{\sum_{b\in\B} \overline\chi(b)} \notag \\
&= A_\B(s) \prod_{p\mid q} (1-p^{-s})^B \prod_{\substack{\chi\mod q \\ \chi\ne\chi_0}} L(s,\chi)^{\sum_{b\in\B} \overline\chi(b)} \label{G is A times Ls}
\end{align}
can be analytically continued to the zero-free region~\eqref{combined bound} of these Dirichlet $L$-functions, which implies that $G_\B(s)$ itself can be analytically continued to this region.
\end{proof}

The other requirement of $F(s)$ having the property $\P(\beta;c_0,\frac12,M)$ is an upper bound for the related function $G_\B(s)$ inside the zero-free region~\eqref{combined bound}, which the following lemma furnishes.

\begin{lemma} \label{property 2 lemma}
There exists an absolute constant $C>0$ such that, in the region~\eqref{region of a.c.}, we have the bound $|G_\B(s)| \le M(1+|\tau|)^{1/2}$ with $M \ll q^{1/2} \log^2 q \cdot (C\underline B)^{\underline B}$.
\end{lemma}

\begin{proof}
The logarithm of the expression~\eqref{gonna be crude} can be bounded by
\begin{align*}
\log |A_\B(s)| &= \sum_{p \le q^{1/\sigma}} \log \bigg( 1 + O\bigg( \frac{\phi(q)^2}{p^{2\sigma}} \bigg) \bigg) + \sum_{p > q^{1/\sigma}} \log\bigg( 1 + O\bigg( \frac{\phi(q)^2}{p^{2\sigma}} \bigg) \bigg) \\
&\ll \sum_{p \le q^{1/\sigma}} \log \phi(q)^2 + \sum_{p > q^{1/\sigma}} \frac{\phi(q)^2}{p^{2\sigma}} \ll \pi(q^{1/\sigma}) \log q + \frac{q^2}{( q^{1/\sigma})^{2\sigma-1}} \ll q^{1/\sigma}
\end{align*}
uniformly for $\sigma\ge\frac34$ say. In particular, when $\sigma \ge 1-1/(3\log q)$, we obtain $\log |A_\B(s)| \ll q$.

If $\chi\mod q$ is not an exceptional character, then by the first case of equation~\eqref{combined bound} and Lemma~\ref{w^z lemma},
\begin{align*}
\log \big| L(s,\chi)^{\sum_{b\in\B} \overline\chi(b)} \big| \le \big( \log\log^+q|\tau| + O(1) \big) \bigg| \sum_{b\in\B} \overline\chi(b) \bigg| \le \underline B \big( \log\log^+q|\tau| + O(1) \big).
\end{align*}
(Recall that $\underline B=\min\{B,\phi(q)-B\}$; the bound $|\sum_{b\in\B} \overline\chi(b)| \le B$ is obvious, while the bound $|\sum_{b\in\B} \overline\chi(b)| \le \phi(q)-B$ follows, since $\chi\ne\chi_0$, from the consequence
\[
\sum_{b\in\B} \overline\chi(b) = -\sum_{b\in(\Z/q\Z)^\times\setminus \B} \overline\chi(b)
\]
of orthogonality.)
Similarly, if $\chi\mod q$ is an exceptional character, then by the second case of equation~\eqref{combined bound} and Lemma~\ref{w^z lemma},
\begin{align*}
\log \big| L(s,\chi)^{\sum_{b\in\B} \overline\chi(b)} \big| \le \underline B \big( \tfrac12\log q + 3 \log\log^+q|\tau| + O(1) \big).
\end{align*}
Since there can be at most one exceptional character, the product in equation~\eqref{G is A times Ls} satisfies the bound
\begin{align*}
\log\bigg| \prod_{\substack{\chi\mod q \\ \chi\ne\chi_0}} L(s,\chi)^{\sum_{b\in\B} \overline\chi(b)} \bigg| \le \underline B \Big( \phi(q) \big( \log\log^+q|\tau| + O(1) \big) + \big( \tfrac12\log q + 3 \log\log^+q|\tau| + O(1) \big) \Big),
\end{align*}
and equation~\eqref{G is A times Ls} becomes, for $\sigma>1-1/(3\log q)$,
\begin{align} \label{gonna exponentiate}
\phi(q) \log |G_\B(s)| \le q + \underline B \big( (\phi(q)+3) \big( \log\log^+q|\tau| + O(1) \big) + \tfrac12\log q \big).
\end{align}

We know that $q/\phi(q) < 2\log\log q$ when~$q$ is sufficiently large, and thus $e^{q/\phi(q)} \ll \log^2q$. Therefore, dividing both sides of equation~\eqref{gonna exponentiate} by $\phi(q)$ and exponentiating yields, for some absolute constants~$C_1$ and~$C_2$,
\begin{align}
|G_\B(s)| &\ll e^{q/\phi(q)} \big( (C_1\log^+q|\tau|)^{(1+3/\phi(q))} e^{(\log q)/2\phi(q)} \big)^{\underline B} \notag \\
&\ll \log^2 q \cdot \big( C_2 (\log^+q|\tau|)^{(1+3/\phi(q))} \big)^{\underline B}. \label{e/q}
\end{align}
A routine calculus exercise shows that the quotient of the right-hand side with $|\tau|^{1/2}$, for $|\tau|\ge e/q$, is maximized at $|\tau| = e^{2\underline B(1+3/\phi(q))}/q$, from which we obtain
\[
|G_\B(s)| \ll \log^2 q \cdot \Big( C_2 \big( \tfrac2e (1+\tfrac3{\phi(q)})\underline B \big)^{1+3/\phi(q)} \Big)^{\underline B} q^{1/2} |\tau|^{1/2} \ll q^{1/2} \log^2 q \cdot (C\underline B)^{\underline B} (1+|\tau|)^{1/2}
\]
for some absolute constant~$C$. On the other hand, when $|\tau|\le e/q$, the right-hand side of equation~\eqref{e/q} simply equals $(C_2)^{\underline B}\log^2q$, which justifies this final estimate for those values of~$\tau$ as~well.
\end{proof}

We remark in passing that all of the constants in these proofs are effectively computable. It might seem that we could obtain a better upper bound ineffectively for~$M$ in this proposition by using Siegel's theorem on exceptional zeros in Lemma~\ref{exceptional bound} and Proposition~\ref{zerofree prop}; this is a mirage, however, as the expression $q^{1/2}\log^2q$ in the estimate for~$M$ actually arises by accident from completely different sources.

We are now able to count the integers in $\cN_\B$ using the Selberg--Delange method, at least when~$\B$ is small enough. It is crucial for our application that we can produce an asymptotic formula for this counting function that has some uniformity in~$q$.

\begin{theorem} \label{general B thm}
There exist positive absolute constants~$\alpha$, $\eta$, and~$C$ such that, uniformly for $\log x \ge \alpha q^{1/2}\log^2 q$,
\begin{multline*}
\#\{ n\le x\colon n\in\cN_\B\} \\
= \frac x{(\log x)^{1-\beta}} \bigg( \frac{G_\B(1)}{\Gamma(\beta)} + O\Big( (C\underline B)^{\underline B} q^2(\log q)^8 \big( e^{-\eta\sqrt{(\log x)/q^{1/2}\log^2 q}} + (\log x)^{-1} \big) \Big) \bigg),
\end{multline*}
where $G_\B(s) = \zeta(s)^{-\beta} \prod_{p\in\B} \big( 1-p^{-s} \big)^{-1}$ has an analytic continuation to a neighborhood of $s=1$.

In particular, uniformly for $q \le (\log x)^{1/3}$,
\begin{align} \label{uniform B formula}
\#\{ n\le x\colon n\in\cN_\B\} &= \frac x{(\log x)^{1-\beta}} \bigg( \frac{G_\B(1)}{\Gamma(\beta)} + O\big( (C\underline B)^{\underline B} (\log x)^{-1/4} \big) \bigg).
\end{align}
\end{theorem}

\begin{proof}
The result is a direct application of Corollary~\ref{SD corollary} with $F(s) = F_\B(s)$, which by Lemmas~\ref{property 1 lemma} and~\ref{property 2 lemma} has the property $\P(\beta;c_0,\frac12,M)$ with the parameters $c_0=\eta/q^{1/2}\log^2q$ and $M \ll q^{1/2} \log^2 q \cdot (C\underline B)^{\underline B}$.
\end{proof}

\begin{remark} \label{B B' remark}
If $\B$ and $\B'$ partition the reduced residue classes\mod q, then clearly $\beta+\beta'=1$ and
\begin{align*}
G_\B(s) G_{\B'}(s) &= \zeta(s)^{-\beta} \prod_{p\in\B} \bigg( 1-\frac1{p^s} \bigg)^{-1} \zeta(s)^{-\beta'} \prod_{p\in\B'} \bigg( 1-\frac1{p^s} \bigg)^{-1} \\
&= \zeta(s)^{-1} \prod_{p\nmid q} \bigg( 1-\frac1{p^s} \bigg)^{-1} = \prod_{p\mid q} \bigg( 1-\frac1{p^s} \bigg);
\end{align*}
in particular, $G_\B(1) G_{\B'}(1) = \phi(q)/q$. Note also that $\underline B = \underline B'$ in this situation. Consequently, one can derive the asymptotic formula for $\#\{ n\le x\colon n\in\cN_{\B'}\}$ directly from the asymptotic formula for $\#\{ n\le x\colon n\in\cN_\B\}$ from Theorem~\ref{general B thm}, with the same error term.
\end{remark}

With an eye towards future applications, we exhibit one last variant of the above results, where the set~$\B$ of primes (defined as in Notation~\ref{B beta notation}) in modified by the removal of finitely many primes and the insertion of finitely many other primes. To fix notation, let $\cR=\{b_1,\dots,b_k\}$ be a finite set of primes in~$\B$, and let $\cI = \{p_1,\dots,p_j\}$ be a finite set of primes that are not in~$\B$. We will establish an asymptotic formula for the counting function of $\cN_{\cI\cup\B\setminus\cR} = \{n\in\N\colon p\mid n\Rightarrow p\in\cI\cup\B\setminus\cR\}$, the set of integers all of whose prime factors lie in $\cI\cup\B\setminus\cR$.
To this set of integers, define the associated Dirichlet series
\begin{equation} \label{I and B not R products}
F_{\cI\cup\B\setminus\cR}(s) = \sum_{n\in\cN_\B} \frac1{n^s} = \prod_{p\in\cI\cup\B\setminus\cR} \bigg( 1-\frac1{p^s} \bigg)^{-1} = F_\B(s) \prod_{i=1}^j \bigg( 1-\frac1{p_i^s} \bigg)^{-1} \prod_{i=1}^k \bigg( 1-\frac1{b_k^s} \bigg).
\end{equation}
As before we set $G_{\cI\cup\B\setminus\cR}(s) = F_{\cI\cup\B\setminus\cR}(s) \zeta(s)^{-\beta}$.

Each of the finitely many factors on the right-hand side of equation~\eqref{I and B not R products} is analytic for $\Re s>0$, and so Lemma~\ref{property 1 lemma}
implies that $G_{\cI\cup\B\setminus\cR}(s)$ has an analytic continuation to the region~\eqref{region of a.c.}. It also easily follows from Lemma~\ref{property 2 lemma} that there exists an absolute constant $C>0$ such that, in the region~\eqref{region of a.c.}, we have the bound $|G_{\cI\cup\B\setminus\cR}(s)| \le M(1+|\tau|)^{1/2}$ with
\begin{equation} \label{IBR error products}
M \ll q^{1/2} \log^2 q \cdot (C\underline B)^{\underline B} \prod_{i=1}^j \bigg( 1-\frac1{p_i^{\sigma(q)}} \bigg)^{-1} \prod_{i=1}^k \bigg( 1 + \frac1{b_k^{\sigma(q)}} \bigg),
\end{equation}
where
\begin{equation*}
\sigma(q)=1-1/(81q^{1/2}\log^2 q) \ge 0.99.
\end{equation*}
It follows that each factor in the products on the right-hand side of equation~\eqref{IBR error products} is at most $2$ (with a single exception if some $p_i$ equals $2$, in which case the corresponding factor is at most~$3$).
The proof of Theorem~\ref{general B thm} then yields the variant:

\begin{theorem} \label{general IBR thm}
There exist positive absolute constants~$\alpha$, $\eta$, and~$C$ such that, uniformly for $\log x \ge \alpha q^{1/2}\log^2 q$,
\begin{multline*}
\#\{ n\le x\colon n\in\cN_{\cI\cup\B\setminus\cR}\} = \frac x{(\log x)^{1-\beta}} \bigg\{ \frac{G_\B(1)}{\Gamma(\beta)} \prod_{i=1}^j \bigg( 1-\frac1{p_i} \bigg)^{-1} \prod_{i=1}^k \bigg( 1-\frac1{b_k} \bigg) \\
+ O\Big( (C\underline B)^{\underline B} 2^{j+k} q^2(\log q)^8 \big( e^{-\eta\sqrt{(\log x)/q^{1/2}\log^2 q}} + (\log x)^{-1} \big) \Big) \bigg\}.
\end{multline*}
\end{theorem}

\section{Integers with all prime factors $1\mod q$, uniformly in~$q$} \label{special 1}

We now apply the previous theorem to the classical counting problem for integers all of whose prime factors are congruent to $1\mod q$.
Define
\[
J_q(x) = \#\{ n\le x\colon p\mid n\implies p\equiv1\mod q\}
\]

\begin{proposition} \label{B=1 prop}
For any positive even integer~$q$, uniformly for $q \le (\log x)^{1/3}$,
\[
J_q(x) = \frac{G_q x}{(\log x)^{1-1/\phi(q)}} + O\bigg( \frac x{(\log x)^{5/4-1/\phi(q)}} \bigg),
\]
where $G_q$ is defined in equation~\eqref{Gq def}.
\end{proposition}

\begin{proof}
We invoke the last assertion of Theorem~\ref{general B thm} with $\B$ equal to the set of integers congruent to $1\mod q$, so that $\beta=1/\phi(q)$ and $\underline B=B=1$; for this proof we abuse notation and write simply $\B=\{1\}$. Our only task is to evaluate the leading constant $G_{\{1\}}(1)/\Gamma(1/\phi(q))$. By equation~\eqref{G is A times Ls}, we see that
\begin{equation} \label{expanding G1}
G_{\{1\}}(s)^{\phi(q)} = A_{\{1\}}(s) \prod_{p\mid q} (1-p^{-s})^1 \prod_{\substack{\chi\mod q \\ \chi\ne\chi_0}} L(s,\chi)^{\overline\chi(1)} = A_{\{1\}}(s) \prod_{p\mid q} (1-p^{-s}) \prod_{\substack{\chi\mod q \\ \chi\ne\chi_0}} L(s,\chi),
\end{equation}
where by equation~\eqref{AB def},
\begin{align*}
A_{\{1\}}(s) &= F_{\{1\}}(s)^{\phi(q)} \prod_{\chi\mod q} L(s,\chi)^{-\overline\chi(1)} \\
&= \prod_{p\equiv1\mod q} \bigg( 1-\frac1{p^s} \bigg)^{-\phi(q)} \prod_{\chi\mod q} \prod_p \bigg( 1-\frac{\chi(p)}{p^s} \bigg) \\
&= \prod_{\chi\mod q} \prod_{p\not\equiv1\mod q} \bigg( 1-\frac{\chi(p)}{p^s} \bigg).
\end{align*}
For each prime~$p$ not dividing~$q$ in the product, there are exactly $\phi(q)/\ord_q(p)$ characters that send~$p$ to any given $\ord_q(p)$th root of unity. Hence
\[
\prod_{\chi \mod{q}} \bigg( 1-\frac{\chi(p)}{p^s} \bigg) = \prod_{k=1}^{\ord_q(p)} \bigg( 1- \frac{e^{2\pi ik/\ord_q(p)}}{p^s} \bigg)^{\phi(q)/\ord_q(p)} = \bigg(1-\frac{1}{(p^s)^{\ord_q(p)}} \bigg)^{\phi(q)/\ord_q(p)},
\]
so that
\begin{equation} \label{A1}
A_{\{1\}}(s) = \prod_{\substack{p\nmid q \\ p\not\equiv1\mod q}} \bigg(1-\frac{1}{p^{s\ord_q(p)}} \bigg)^{\phi(q)/\ord_q(p)}
\end{equation}
Inserting this identity into equation~\eqref{expanding G1} and taking $s=1$ yields
\begin{align*}
G_{\{1\}}(1) &= \bigg\{ \prod_{\substack{p\nmid q \\ p\not\equiv1\mod q}} \bigg(1-\frac{1}{p^{\ord_q(p)}} \bigg)^{\phi(q)/\ord_q(p)} \prod_{p\mid q} (1-p^{-1}) \prod_{\substack{\chi\mod q \\ \chi\ne\chi_0}} L(1,\chi) \bigg\}^{1/\phi(q)} \\
&= \prod_{\substack{p\nmid q \\ p\not\equiv1\mod q}} \bigg(1-\frac{1}{p^{\ord_q(p)}} \bigg)^{1/\ord_q(p)} \bigg\{ \frac{\phi(q)}q \prod_{\substack{\chi\mod q \\ \chi\ne\chi_0}} L(1,\chi) \bigg\}^{1/\phi(q)};
\end{align*}
together with the factor $1/\Gamma(1/\phi(q))$ from equation~\eqref{uniform B formula}, this comprises the constant~$G_q$ defined in equation~\eqref{Gq def} as required.
\end{proof}

\begin{remark}
There are many equivalent ways to write the constant~$G_q$ defined in equation~\eqref{Gq def}; we note, for example, that the first parenthetical factor is actually equal to a residue of a Dedekind zeta-function, which can be evaluated using the analytic class number formula (see~\cite[Chapter~VII, Corollary~5.11]{Neukirch} for example) and the formula for the discriminant of a cyclotomic field (see~\cite[Proposition 2.7]{Washington} for example):
\begin{align*}
\frac{\phi(q)}q \prod_{\substack{\chi\mod q \\ \chi \ne \chi_0}} L(1,\chi) &= \mathop{\rm Res}_{s=1} \zeta_{\Q(e^{2\pi i/q})}(s) \\
&= \frac{(2\pi)^{\phi(q)/2} \gcd(q,2) \mathop{\rm Reg}_{\Q(e^{2\pi i/q})} h(\Q(e^{2\pi i/q})) \prod_{p\mid q} p^{\phi(q)/2(p-1)}}{2q^{\phi(q)/2+1}}
\end{align*}
where $\mathop{\rm Reg}(K)$ and $h(K)$ denote the regulator and class number of~$K$. The form we have chosen has the advantage of being easy to estimate, as the following proposition will show.

A very similar constant appears in the counting function of those integers~$n$ for which $\phi(n)$ is not divisible by a given prime~$q$ (see~\cite{SW,FLM}, and in particular note that our calculation~\eqref{A1} has a close counterpart in~\cite{FLM}). This is to be expected, as such integers are essentially those free of prime factors $p\equiv1\mod q$, which is the case $\B'=\{2,\dots,q-1\}$ of Theorem~\ref{general B thm}; and by Remark~\ref{B B' remark}, the leading constant in that asymptotic formula is closely related to the leading constant~\eqref{Gq def} for $\B=\{1\}$ (in the notation of the proof of Proposition~\ref{B=1 prop}).
\end{remark}

\begin{proposition} \label{Gq estimate prop}
The constant $G_q$ defined in equation~\eqref{Gq def} satisfies $\displaystyle G_q \ll \frac{\log q}{\phi(q)}$.
\end{proposition}

\begin{proof}
The factor $\phi(q)/q$ is at most~$1$, as is the second product over primes. Since $\Gamma(s)$ has a pole at $s=0$, we have $1/\Gamma(s) \ll |s|$ for $|s|\le\frac12$, say, and in particular $1/\Gamma(1/\phi(q)) \ll 1/\phi(q)$. Finally, it is well known (see~\cite[Theorem~11.4]{MV} for example) that $L(1,\chi) \ll \log q$. Combining these estimates establishes the proposition.
\end{proof}

\section{Counting $n$ for which $q \mid \lambda_1(n)$} \label{algebra to number theory sec}

We return now to invariant factors, beginning with characterizations of the least invariant factor $\lambda_1(n)$ in terms of the prime factorization of~$n$.

\begin{lemma} \label{abelian group lemma}
Let $q$ and $m_1,\dots,m_\ell$ be positive integers.
Suppose that $G \cong \Z/m_1\Z \times \cdots \times \Z/m_\ell\Z$, where $\gcd(m_1,\dots,m_\ell)>1$. Then $q\mid \lambda_1(G)$ if and only if $q\mid m_j$ for each $1\le j\le\ell$.
\end{lemma}

\noindent Note that the greatest common divisor condition is necessary, as shown by the example $G = \Z/30\Z \times \Z/30\Z \cong \Z/6\Z \times \Z/10\Z \times \Z/15\Z$ for which $\lambda_1(G) = 30$.

\begin{proof}
It suffices to prove the lemma when~$q$ is a prime power. Note that the assertion $q\mid \lambda_1(G)$ is equivalent to the existence of a subgroup of~$G$ of the form $(\Z/q\Z)^m$, where~$m$ is the length of the invariant factor decomposition of~$G$; on the other hand, the assertion that $q\mid m_j$ for each $1\le j\le\ell$ is equivalent to the existence of a subgroup of~$G$ of the form $(\Z/q\Z)^\ell$. Therefore it suffices to show that the length of the invariant factor decomposition of~$G$ equals~$\ell$.

By hypothesis, there exists a prime~$p$ dividing $\gcd(m_1,\dots,m_\ell)$, which implies that $\Z/m_1\Z \times \cdots \times \Z/m_\ell\Z$ has a subgroup isomorphic to $(\Z/p\Z)^\ell$. In particular, the length of the invariant factor decomposition of~$G$ is at least~$\ell$.

On the other hand, for any finite abelian group~$H$, the length of the invariant factor decomposition of~$H$ is the maximum, over all primes~$p$ dividing the cardinality of~$H$, of the length of the $p$-Sylow subgroup of~$H$. But the $p$-Sylow subgroup of $\Z/m_1\Z \times \cdots \times \Z/m_\ell\Z$ is simply $\Z/p^{\nu_1} \times \cdots \times \Z/p^{\nu_\ell}$ where each~$\nu_j$ is the power of~$p$ in the prime factorization of~$m_j$; in particular, this length is at most~$\ell$.
\end{proof}

The following useful corollary is an immediate consequence of Lemma~\ref{abelian group lemma}:

\begin{corollary} \label{special 2 corollary}
If $m_1,\dots,m_\ell$ are even integers, then the least invariant factor of $\Z/2\Z \times \Z/m_1\Z \times \cdots \times \Z/m_\ell\Z$ equals~$2$.
\end{corollary}

The next proposition completely characterizes the integers~$n$ counted by~$D_q(x)$.
We define~$q_{(p)}$ to be the largest divisor of~$q$ that is not divisible by~$p$, or equivalently $q_{(p)} = q/p^r$ where $p^r$ is the exact power of~$p$ dividing~$q$.

\begin{proposition} \label{lif to factorization prop}
Fix an even integer $q\ge4$. For any positive integer~$n$, the least invariant factor of~$M_n$ is a multiple of~$q$ if and only if all of the following conditions hold:
\begin{enumerate}
\item for primes $p\nmid q$: if $p\mid n$ then we must have $p\equiv1\mod q$;
\item for odd primes $p\mid q$ such that $q_{(p)} \mid (p-1)$: either $p\nmid n$ or $p^{r+1}\mid n$; 
\item for odd primes $p\mid q$ such that $q_{(p)} \nmid (p-1)$: we must have $p\nmid n$;
\item $4\nmid n$.
\end{enumerate}
\end{proposition}

\noindent Note that there can be at most one prime~$p$ satisfying the assumptions of condition~(2) (see the proof of Proposition~\ref{Hq prop} below). The reader is invited to form the intuition that~(1) is the most important condition, so that the order of magnitude of $D_q(x)$ is the same as the order of magnitude of~$J_q(x)$ from the previous section, while the conditions~(2)--(4) at the finitely many primes dividing~$q$ are more minor and affect merely the leading constant in the asymptotic formula.

\begin{proof}
First, assume that~$n$ is odd. Write the factorization~$n=p_1^{\nu_1}\cdots p_m^{\nu_m}$. Since~$n$ is odd, each $\phi(p_j^{\nu_j})$ is even, and each $\Z/p_j^{\nu_j}\Z$ is isomorphic to $\Z/\phi(p_j^{\nu_j})\Z$. By the Chinese remainder theorem, $M_n\cong \Z/\phi(p_1^{\nu_1})\Z \times \cdots \times \Z/\phi(p_m^{\nu_m})\Z$ where $2\mid\gcd\big(\phi(p_1^{\nu_1}),\dots,\phi(p_m^{\nu_m})\big)$. Therefore, by Lemma~\ref{abelian group lemma}, the least invariant factor of~$M_n$ is a multiple of~$q$ if and only if $q\mid\phi(p^\nu) = p^{\nu-1}(p-1)$ for all $p^\nu\|n$. The lemma now follows for odd~$n$ upon examining the three given cases for the relationship between~$p$ and~$q$.

Next, assume that~$n$ is twice an odd number; then $M_n \cong M_{n/2}$, and the lemma follows for these~$n$ by the argument above applied to $n/2$.

Finally, assume that $4\mid n$. Write the factorization~$n=2^r p_1^{\nu_1}\cdots p_m^{\nu_m}$. By the Chinese remainder theorem, $M_n$ is isomorphic to $\Z/2Z \times \Z/\phi(p_1^{\nu_1})\Z \times \cdots \times \Z/\phi(p_m^{\nu_m})\Z$ if $r=2$ and isomorphic to $\Z/2Z \times \Z/2^{r-2}\Z \times \Z/\phi(p_1^{\nu_1})\Z \times \cdots \times \Z/\phi(p_m^{\nu_m})\Z$ if $r\ge3$; in either case, Corollary~\ref{special 2 corollary} implies that the least invariant factor of~$M_n$ equals~$2$. In other words, if the least invariant factor of $M_n$ is a multiple of~$q$ then $4\nmid n$, completing the proof of the proposition.
\end{proof}

\noindent We remark that an examination of the proof confirms that $\lambda_1(n)$ is even for all $n\ge3$.

We are now ready to establish an asymptotic formula for the counting function $D_q(x)$ from Definition~\ref{Dq def}.

\begin{proposition} \label{Hq prop}
Uniformly for even integers $4 \le q\le(\log x)^{1/3}$,
\[
D_q(x) = \frac{H_qx}{(\log x)^{1-1/\phi(q)}} + O\bigg( \frac x{(\log x)^{5/4-1/\phi(q)}} \bigg),
\]
where $H_q$ is defined in equation~\eqref{Hq def}.
In particular, when $q \ll (\log x)^{1/4}$,
\begin{equation} \label{Dq first bound}
D_q(x) \ll \frac{\log q}{\phi(q)} \frac{x}{(\log x)^{1-1/\phi(q)}}.
\end{equation}
\end{proposition}

\begin{proof}
Define $D^1_q(x)$ to be the number of odd positive integers $n\le x$ such that $q$ divides the least invariant factor of~$n$.
By condition~(4) of Lemma~\ref{lif to factorization prop}, every even integer counted by $D_q(x)$ is of the form $2m$ where $m$ is counted by $D^1_q(\frac x2)$. Therefore $D_q(x) = D^1_q(x) + D^1_q(\frac x2)$, and so it suffices to show that
\begin{multline} \label{pre 3/2}
D^1_q(x) = \frac{G_qx}{(\log x)^{1-1/\phi(q)}} \begin{cases}
1 + 1/p^r(p-1), &\text{if there exists an odd prime power} \\
&\qquad\quad p^r\|q \text{ such that } q/p^r \mid (p-1), \\
1, &\text{otherwise}
\end{cases} \\
+ O\bigg( \frac x{(\log x)^{5/4-1/\phi(q)}} \bigg).
\end{multline}
If there is no prime power $p^r$ satisfying the condition on the right-hand side of equation~\eqref{pre 3/2}, then by conditions~(1) and~(3) of Lemma~\ref{lif to factorization prop} the integers counted by $D^1_q(x)$ are precisely those integers all of whose prime factors are congruent to 1\mod q, the number of which is $J_q(x)$ by definition; hence the asymptotic formula~\eqref{pre 3/2} is exactly Proposition~\ref{B=1 prop} in this case.

Next we argue why there can be at most one prime power~$p^r$ satisfying the condition on the right-hand side of equation~\eqref{pre 3/2}: if $p_1^{r_1}$ and $p_2^{r_2}$ were to satisfy these conditions (with $p_1\ne p_2$), then
\begin{equation} \label{only one p}
p_1p_2 \mid q = p_2^{r_2} (q/p_2^{r_2}) \mid (q/p_1^{r_1}) (q/p_2^{r_2}) \mid (p_1-1)(p_2-1) < p_1p_2,
\end{equation}
a contradiction.

The only remaining case to consider is if there is a (unique) odd prime power $p^r$ satisfying the condition on the right-hand side of equation~\eqref{Hq def}. In this case, conditions~(1)--(3) of Lemma~\ref{lif to factorization prop} imply that
\[
D^1_q(x) = D^1_{q,0}(x) + \sum_{\ell\ge r+1} D^1_{q,\ell}(x),
\]
where $D^1_{q,\ell}(x)$ denotes the number of odd integers $n\le x$ of the form $n=p^\ell m$ where all prime factors of $m$ are congruent to 1\mod q; this quantity is exactly $J_q(x/p^\ell)$. Since the number of integers up to~$x$ that are divisible by $p^L$ is trivially at most $x/p^L$ for any positive integer~$L$, we may truncate the above sum and apply Proposition~\ref{B=1 prop} to obtain
\begin{align*}
D^1_q(x) &= J_q(x) + \sum_{\ell=r+1}^{L-1} J_q\bigg( \frac x{p^\ell} \bigg) + O\bigg( \frac x{p^L} \bigg) \\
&= \frac{G_q x}{(\log x)^{1-1/\phi(q)}} + O\bigg( \frac x{(\log x)^{5/4-1/\phi(q)}} \bigg) \\
&\qquad{}+ \sum_{\ell=r+1}^{L-1} \bigg( \frac{G_q x/p^\ell}{(\log (x/p^\ell))^{1-1/\phi(q)}} + O\bigg( \frac{x/p^\ell}{(\log (x/p^\ell))^{5/4-1/\phi(q)}} \bigg) \bigg) + O\bigg( \frac x{p^L} \bigg).
\end{align*}
If we choose~$L$ so that $p^{L-1} < \exp(\sqrt{\log x}) \le p^L$, then $\log(x/p^\ell) = (\log x)(1+O(1/\sqrt{\log x}))$ for each~$\ell$ in the sum, and hence the above expression can be simplified to
\begin{align*}
D^1_q(x) &= \bigg( G_q + O\bigg( \frac 1{(\log x)^{1/4}} \bigg) \bigg) \frac x{(\log x)^{1-1/\phi(q)}} \bigg( 1 + \sum_{\ell=r+1}^{L-1} \frac1{p^\ell} \bigg) + O\bigg( \frac x{p^L} \bigg) \\
&= \bigg( G_q + O\bigg( \frac 1{(\log x)^{1/4}} \bigg) \bigg) \frac x{(\log x)^{1-1/\phi(q)}} \bigg( 1 + \frac{1}{p^r(p-1)} - \frac1{p^{L-1}(p-1)} \bigg) + O\bigg( \frac x{p^L} \bigg) \\
&= \bigg( G_q + O\bigg( \frac 1{(\log x)^{1/4}} \bigg) \bigg) \frac x{(\log x)^{1-1/\phi(q)}} \bigg( 1 + \frac{1}{p^r(p-1)} \bigg) + O\bigg( \frac x{\exp(\sqrt{\log x})} \bigg)
\end{align*}
(since $G_q \ll 1$ by Proposition~\ref{Gq estimate prop}), which is enough to establish equation~\eqref{pre 3/2} in this final case.
\end{proof}

It is also possible to prove Proposition~\ref{Hq prop} directly using the Selberg--Delange method, inserting relevant Euler factors at the primes dividing~$q$. However, we have chosen our proof to proceed via Theorem~\ref{general B thm} in the hopes that developing this general infrastructure will have value beyond this paper.

\section{Counting $n$ for which $\lambda_1(n) = q$} \label{second thm sec}

The goal of this section is to obtain an asymptotic formula for $E_q(x)$, the counting function of those integers~$n$ for which $\lambda_1(n)=q$ (and thereby establish Theorem~\ref{Eq theorem}), from the asymptotic formula established in the previous section for $D_q(x)$, the counting function of those integers~$n$ for which $q\mid \lambda_1(n)$. We proceed essentially by inclusion--exclusion (encoded, as usual, by the M\"obius mu-function). For technical reasons, however, we work with slight variants of the functions $E_q(x)$ and $D_q(x)$ in which we insist that the integers being counted do not have few prime factors. The following definition, in which $\omega(n)$ as usual denotes the number of distinct prime factors of~$n$, can be compared with Definitions~\ref{Eq def} and~\ref{Dq def}:

\begin{definition}
Given an integer parameter $b\ge4$, define
\begin{align*}
E_q^*(x;b) &= \#\{ n \le x\colon \lambda_1(n) = q \text{ and } \omega(n) > b \} \\
D_q^*(x;b) &= \#\{ n \le x\colon q \mid \lambda_1(n) \text{ and } \omega(n) > b \}
\end{align*}
\end{definition}

\begin{lemma} \label{Eq with M lemma}
For all integers $q,M,b\ge4$ and all real numbers $x\ge e^e$,
\[
E_q(x) = D_q(x) + O_b \bigg(\sum_{1 <m<M} D_{mq}(x) + \sum_{m \ge M} D_{mq}^*(x;b) +  \frac{Mx(\log \log x)^{b-1}}{\log x} \bigg).
\]
\end{lemma}

\begin{proof}
It is obvious from the definitions of $D_q^*$ and $E_q^*$ that
$D_q^*(x;b) = \sum_{m=1}^\infty E_{mq}^*(x;b)$; a version of M\"obius inversion then implies that $E_q^*(x;b) = \sum_{m=1}^\infty \mu(m) D_{mq}^*(x;b)$. (Both series are actually finite sums for any given value of~$x$, so there are no issues of convergence. We note that these identities hold even for $q=2$.)
We write
\begin{align} \label{Eq* first formula}
E_q^*(x;b) &= D_q^*(x;b) + O \bigg( \sum_{1 <m<M} D_{mq}^*(x;b) + \sum_{m \ge M} D_{mq}^*(x;b) \bigg).
\end{align}
Using known estimates for the number of integers with a constant number of prime factors (see~\cite[Section~7.4]{MV} for example),
\begin{align*}
0 \le E_q(x) - E_q^*(x;b) &= \#\{ n \le x\colon \lambda_1(n) = q \text{ and } \omega(n) \le b \} \\
&\le \#\{ n\le x\colon \omega(n) \le b \} \\
&\ll_b \sum_{k \le b} \frac{x}{\log x} \frac{(\log \log x)^{k-1}}{(k-1)!} \ll_b \frac{x(\log \log x)^{b-1}}{\log x}.
\end{align*}
The same argument gives $0 \le D_{mq}(x) - D_{mq}^*(x;b) \ll_b {x(\log \log x)^{b-1}}/{\log x}$ uniformly in~$mq$, and thus the formula~\eqref{Eq* first formula} becomes
\begin{align*}
E_q(x) &= E_q^*(x;b) + O_b \bigg( \frac{x(\log \log x)^{b-1}}{\log x} \bigg) \\
&= D^*_q(x) + O_b \bigg(\sum_{1 <m<M} D_{mq}^*(x;b) + \sum_{m \ge M} D_{mq}^*(x;b) +  \frac{x(\log \log x)^{b-1}}{\log x} \bigg) \\
&= D_q(x) + O_b \bigg(\sum_{1 <m<M} D_{mq}(x) + \sum_{m \ge M} D_{mq}^*(x;b) +  \frac{Mx(\log \log x)^{b-1}}{\log x} \bigg)
\end{align*}
as desired. (A more careful argument could remove the factor of~$M$ from the last error term, but this version suffices for our purposes.)
\end{proof}

We already have the upper bound~\eqref{Dq first bound} that can be applied to $D_{mq}(x)$ for small~$m$; we now complement that with an upper bound for $D_{mq}^*(x;b)$ that improves as~$m$ increases.

\begin{lemma} \label{Dmq* lemma}
For $b\ge3$ and $x\ge e^e$,
\[
D_{mq}^*(x;b) \ll_b \frac{x(\log \log x)^{b-1}}{m^{b-2}}.
\]
\end{lemma}

\begin{proof}
First note that by the argument surrounding equation~\eqref{only one p}, with $q$ replaced by $mq$, there can be at most one prime~$p_0$ such that $(mq)_{(p_0)} \mid (p_0-1)$. Suppose that~$n$ is counted by $D_{mq}^*(x;b)$. We know that $\omega(n)\ge b+1$, and so we can choose $b-1$ distinct odd primes $p_1,\dots,p_{b-1}$ dividing~$n$ (and hence not exceeding~$x$), none of which is equal to $p_0$ if it exists. Of course the number of $n\le x$ divisible by these $b-1$ primes is at most $x/p_1\cdots p_{b-1}$.

We also know that $mq$ divides $\lambda_1(n)$; by Proposition~\ref{lif to factorization prop} with~$q$ replaced by~$mq$, we must have $p_1\equiv\cdots\equiv p_{b-1}\equiv1\mod{mq}$.
Sorting all such integers~$n$ by the primes $p_1,\dots,p_{b-1}$ (and ignoring the possibility of double-counting), we conclude that
\[
D_{mq}^*(x;b) \le \sum_{\substack{p_1 \le x \\ p_1 \equiv 1 \mod{mq}}} \dots \sum_{\substack{p_{b-1} \le x \\ p_{b-1} \equiv 1 \mod{mq}}} \frac{x}{p_1 \dots p_{b-1}} \le x \bigg( \sum_{\substack{p \le x \\ p \equiv 1 \mod{mq}}} \frac{1}{p} \bigg)^{b-1}.
\]
It follows quickly from the Brun--Titchmarsh theorem and partial summation that for $x\ge e^e$,
\[
\sum_{\substack{p \le x \\ p \equiv 1 \mod{mq}}} \frac{1}{p} \ll \frac{\log \log x}{\phi(mq)};
\]
therefore
\begin{align*}
D_{mq}^*(x;b) &\ll_b x \bigg( \frac{\log \log x}{\phi(mq)} \bigg)^{b-1} \le x \bigg( \frac{\log \log x}{\phi(m)} \bigg)^{b-1}.
\end{align*}
The lemma follows since $\phi(m) \gg_b m^{(b-2)/(b-1)}$.
\end{proof}

We now have the tools we need to assemble our asymptotic formula for $E_q(x)$ for all even integers $q\ge4$.

\begin{proof}[Proof of Theorem~\ref{Eq theorem}]
Beginning with Lemma~\ref{Eq with M lemma}, we input the upper bound~\eqref{Dq first bound} for the terms in the first error-term sum (valid as long as $M\le(\log x)^{1/4}$) and Lemma~\ref{Dmq* lemma} for the terms in the second error-term sum, obtaining
\begin{align*}
E_q(x) &= D_q(x) + O_b \bigg(\sum_{1 <m<M} \frac{\log mq}{\phi(mq)} \frac{x}{(\log x)^{1-1/\phi(mq)}} + \sum_{m \ge M} \frac{x(\log \log x)^{b-1}}{m^{b-2}} +  \frac{Mx(\log \log x)^{b-1}}{\log x} \bigg) \\
&= D_q(x) + O_{q,b} \bigg( \frac{x}{(\log x)^{1-1/\phi(2q)}} \sum_{1<m<M} \frac{\log m}{\phi(m)} + \frac{x(\log \log x)^{b-1}}{M^{b-3}} +  \frac{Mx(\log \log x)^{b-1}}{\log x} \bigg) \\
&= D_q(x) + O_{q,b} \bigg( \frac{x(\log M)^2}{(\log x)^{1-1/2\phi(q)}} + \frac{x(\log \log x)^{b-1}}{M^{b-3}} +  \frac{Mx(\log \log x)^{b-1}}{\log x} \bigg).
\end{align*}
Choosing $M=\big\lfloor (\log x)^{1/3\phi(q)} \big\rfloor$ and $b=3\phi(q)+2$, we conclude that
\begin{align*}
E_q(x) &= D_q(x) + O_{q} \bigg( \frac{x(\log\log x)^2}{(\log x)^{1-1/2\phi(q)}} + \frac{x(\log \log x)^{3\phi(q)+1}}{(\log x)^{1-1/3\phi(q)}} \bigg) = D_q(x) + O_{q} \bigg( \frac{x(\log\log x)^2}{(\log x)^{1-1/2\phi(q)}} \bigg),
\end{align*}
which suffices to establish the theorem in light of Proposition~\ref{Hq prop}.
\end{proof}

\section{Counting $n$ for which $\lambda_1(n) \ne 2$} \label{main thm sec}

We adapt the method of the previous section to assemble our asymptotic formula for~$T(x)$.

\begin{proof}[Proof of Theorem~\ref{T theorem}]
Because of the identity
\[
T(x) = \lfloor x\rfloor -E_2(x) = D_2(x) + 2 - E_2(x),
\]
we begin by examining the function $E_2(x)$ using the method of the proof of Lemma~\ref{Eq with M lemma}, whose deductions we will use here without specific comment.
We write
\begin{align*}
E_2^*(x;b) &= \sum_{m=1}^\infty \mu(m)D_{2m}^*(x;b) \\
&= D_2^*(x;b) - D_4^*(x;b)-D_6^*(x;b) + O \bigg( D_{10}^*(x;b) + \sum_{7\le m<M} D_{2m}^*(x;b) + \sum_{m\ge M} D_{2m}^*(x;b) \bigg) \\
&= D_2(x)-D_4(x)-D_6(x) \\
&\qquad{}+ O \bigg( D_{10}(x) + \sum_{7\le m<M} D_{2m}(x) + \sum_{m\ge M} D_{2m}^*(x;b) + \frac{M x(\log\log x)^{b-1}}{\log x} \bigg).
\end{align*}
Again we input the upper bound~\eqref{Dq first bound} for the terms in the first error-term sum (valid as long as $M\le(\log x)^{1/4}$) and Lemma~\ref{Dmq* lemma} for the terms in the second error-term sum, obtaining
\begin{align*}
T(x) &= D_2(x)+2-E_2(x) \\
&= D_4(x) + D_6(x) + O \bigg( \frac{x}{(\log x)^{1-1/\phi(10)}} \\
&\qquad{} + \sum_{7\le m<M} \frac{\log 2m}{\phi(2m)} \frac{x}{(\log x)^{1-1/\phi(mq)}}+ \frac{x(\log \log x)^{b-1}}{M^{b-3}} + \frac{M x(\log\log x)^{b-1}}{\log x} \bigg) \\
&= D_4(x) + D_6(x) + O \bigg( \frac{x}{(\log x)^{3/4}} \\
&\qquad{} + \frac{x}{(\log x)^{5/6}} \sum_{m<M} \frac{\log m}{\phi(m)} + \frac{x(\log \log x)^{b-1}}{M^{b-3}} + \frac{M x(\log\log x)^{b-1}}{\log x} \bigg),
\end{align*}
where in the last line we have used the fact that $\phi(mq)\ge6$ for $mq\ge14$. Choosing $M=\lfloor (\log x)^{1/5} \rfloor$ and $b=7$, we conclude that
\begin{align*}
T(x) &= D_4(x) + D_6(x) + O \bigg( \frac{x}{(\log x)^{3/4}} + \frac{x(\log\log x)^2}{(\log x)^{5/6}} + \frac{x(\log \log x)^6}{(\log x)^{4/5}} \bigg) \\
&= D_4(x) + D_6(x) + O \bigg( \frac{x}{(\log x)^{3/4}} \bigg),
\end{align*}
which again suffices to establish the theorem in light of Proposition~\ref{Hq prop}.
\end{proof}

\setcounter{section}{1}
\renewcommand{\thesection}{\Alph{section}}

\section*{Appendix: The Selberg--Delange method with additional uniformity} \label{SD sec}

The goal of this appendix is to establish a version of the Selberg--Delange formula where the dependence on certain parameters has been made explicit. More precisely, Tenenbaum~\cite[Chapter~II.5]{tenenbaum15} has a statement of the Selberg--Delange formula for Dirichlet series~$F(s)$ satisfying certain hypotheses (see Definitions~\ref{property P def} and~\ref{type T def} below) that involve two parameters in particular: $c_0$, a constant involved in the definition of a zero-free-region-shaped domain of analytic continuation for~$F(s)$, and $\delta$, a constant involved in the rate of growth of~$F(s)$ within that domain. In Tenenbaum's work, implicit constants in the formulas are allowed to depend upon~$c_0$ and~$\delta$. Here, we give a version of the theorem that is uniform in those two parameters; see Theorem~\ref{thm_SD} for our most general statement and Corollary~\ref{SD corollary} for an important special case.

We recommend that the reader have Tenenbaum's book available for reference; however, our proof herein is largely self-contained once we quote relevant results from that source. We are essentially following Tenenbaum's proof (except for rendering one error term, in Lemma~\ref{Phi' asymptotic lemma} below, in a form more convenient for our present purposes), although we have reorganized the order of the steps for clarity in this context.

\subsection{Definitions and preliminaries}

All of the definitions in this section are taken directly fom Tenenbaum's book; we also record several simple consequences of those definitions for later use. All of this notation will be used throughout this appendix.

\begin{definition} \label{Z def}
For any complex number $z$, the function $Z(s;z) = \big( (s-1)\zeta(s) \big)^z / s$ is defined on any simply connected domain not containing a zero of $\zeta(s)$. We always assume that its domain of definition contains $[1,\infty)$, and we choose the branch of the complex exponential satisfying $Z(1;z)=1$.

As stated in~\cite[Part~II, Theorem 5.1]{tenenbaum15}, there exist entire functions $\gamma_j(z)$ such that the function $Z(s;z)$ has the Taylor expansion $Z(s;z) = \sum_{j=0}^\infty \frac1{j!} \gamma_j(z) (s-1)^j$, and moreover, on the disk $\{|z|\le A\}$, we have $\frac1{j!} \gamma_j(z) \ll_{A,\ep} (1+\ep)^j$ for any~$\ep>0$.
\end{definition}

\begin{lemma} \label{Z bound lemma}
For any $0<r<1$, we have $Z(s;z) \ll_{A,r} 1$ uniformly for $|s-1|\le r$ and $|z|\le A$.
\end{lemma}

\begin{proof}
Choosing $\ep=\frac{1-r}{2r}$ in the estimate in Definition~\ref{Z def}, we deduce from the Taylor expansion of $Z(s;z)$ that
\[
|Z(s;z)| \le \sum_{j=0}^\infty \bigg| \frac{\gamma_j(z)}{j!} \bigg| |s-1|^j \ll_{A,r} \sum_{j=0}^\infty \bigg( 1+\frac{1-r}{2r} \bigg)^{\!j} r^j = \frac2{1-r} \ll_r1
\]
under the hypotheses of the lemma, as desired.
\end{proof}

We define $\log^+(y) = \max\{ 0,\log y\}$; the following calculus inequality will be helpful later.

\begin{lemma} \label{lil calc lemma}
Given real numbers $A>0$ and $0<\delta\le1$, we have $(1+\log^+y)^A \ll_A \delta^{-A} (1+y)^{\delta/2}$ for all $y\ge0$.
\end{lemma}

\begin{proof}
Since the assertion is trivial for $0\le y\le e^{-1}$, we may assume that $y\ge e^{-1}$. The function $(1+\log t)^A/t^{\delta/2}$ is increasing for $e^{-1}\le t\le e^{2A/\delta-1}$ and decreasing thereafter, and its value at $t=e^{2A/\delta-1}$ equals $2^A e^{{\delta/2}-A} (\frac{A}{\delta })^A \le 2^A e^{1/2-A} (\frac{A}{\delta })^A$; this global maximum value shows that $(1+\log y)^A \ll_A \delta^{-A} y^{\delta/2}$ for all $y\ge e^{-1}$, which is enough to establish the lemma.
\end{proof}

We continue to follow the same convention as Tenenbaum of writing the real and imaginary parts of complex numbers as $s=\sigma+i\tau$.

\begin{definition} \label{property P def}
Let $z$ be a complex number, and let $c_0$, $\delta$, and $M$ be positive real numbers with $\delta\le1$. A Dirichlet series $F(s)$ {\em has the property $\P(z;c_0,\delta,M)$} if the Dirichlet series $G(s;z) = F(s) \zeta(s)^{-z}$ has an analytic continuation to the region $\sigma \ge 1 - c_0/(1+\log^+ |\tau|)$, wherein it satisfies the bound $|G(s;z)| \le M(1+|\tau|)^{1-\delta}$. Note that this inequality implies that $G(s;z) \ll M$ uniformly for $|s-1| \le \min\{1,c_0\}$.
\end{definition}

Define a simply connected domain~$\cD$ and parameters~$c_0$ and~$c$ associated to it by
\begin{equation} \label{D and c def}
\cD = \big\{ 1 - c/(1+\log^+ |\tau|) \le \sigma \le 2 \big\} \setminus (-\infty,1], \quad\text{where }
c=c(c_0)=\tfrac{47}{48}c_0.
\end{equation}
(The fraction~$\frac{47}{48}$ is convenient in one step of the proof of Lemma~\ref{Phi' asymptotic lemma} below and is otherwise insignificant, other than being less than~$1$.)
Notice that if $c_0 \le \frac2{11}$, then $c^{-1} > 5.573412$, which implies~\cite[Theorem 1]{MosTru} that $\zeta(s)\ne0$ for $s\in\cD$.

\begin{lemma} \label{F bound lemma}
Suppose that $F(s)$ has the property $\P(z;c_0,\delta,M)$, where $|z|\le A$ and $c_0 \le \frac2{11}$.
Then $F(s)$ has an analytic continuation to $\cD$ which satisfies
\[
F(s) \ll_A \begin{cases}
\delta^{-A} M(1+|\tau|)^{1-\delta/2}, &\text{if } |\tau| \ge 3, \\
M|s-1|^{-A}, &\text{if } |\tau| \le 3.
\end{cases}
\]
\end{lemma}

\begin{proof}
The analytic continuation follows from the identity $F(s) = G(s;z)\zeta(s)^z$ and the assumed analytic continuation of $G(s;z)$.

In $\{s\in\cD\colon|\tau|\ge3\}$, we have the upper bound $\zeta(s)^z \ll_A (1+\log^+ |\tau|)^A$~\cite[Part~II, equation~(5.10)]{tenenbaum15}, which by our assumption on~$G(s;z)$ implies
\[
F(s) = G(s;z)\zeta(s)^z \ll_A M (1+|\tau|)^{1-\delta} (1+\log^+ |\tau|)^A \ll_{A} \delta^{-A} M(1+|\tau|)^{1-\delta/2}
\]
by Lemma~\ref{lil calc lemma}.

On the other hand, the fact that $\zeta(s) = 1/(s-1) + O(1)$ in any compact neighborhood of $s=1$ implies that 
$|\zeta(s)^z| \ll_A |\zeta(s)|^A \ll_A |s-1|^{-A}$ in $\{s\in\cD\colon|\tau|\le3\}$ (as can be derived from Lemma~\ref{w^z lemma}). Again by assumption, $G(s;z) \ll M (1+|\tau|)^{1-\delta} \ll M$ when $|\tau|\le3$ (since $\delta>0$). These bounds establish the second case of the lemma.
\end{proof}

\begin{definition} \label{type T def}
Let $z$ and $w$ be complex numbers, and let $c_0$, $\delta$, and $M$ be positive real numbers with $\delta\le1$. A Dirichlet series $F(s) = \sum_{n=1}^\infty a_n n^{-s}$ {\em has type $\T(z,w;c_0,\delta,M)$} if it has the property $\P(z;c_0,\delta,M)$ and there exist nonnegative real numbers $b_n$, with $|a_n| \le b_n$ for all $n\ge1$, such that the Dirichlet series $\sum_{n=1}^\infty b_n n^{-s}$ has the property $\P(w;c_0,\delta,M)$ as in Definition~\ref{property P def}.

Note that when the $a_n$ are nonnegative real numbers, if $F(s)$ has the property $\P(z;c_0,\delta,M)$ then it automatically has type $\T(z,z;c_0,\delta,M)$. (We remark that Tenenbaum has ``positive'' in place of ``nonnegative'' in both places in this definition, but strict positivity is never used.)
\end{definition}

\begin{definition} \label{Gamma def}
For $x>e^{1/c}=e^{48/(47c_0)}$, we define
\begin{equation} \label{Phi def}
\Phi(x) = \frac{1}{2\pi i} \int_\Gamma F(s) x^{s+1} \frac{ds}{s(s+1)}
\end{equation}
where $\Gamma$ is a truncated Hankel contour (see Figure~$1$), consisting of a circle around $s=1$ with radius $1/(2\log x)$ and two line segments, above and below the branch cut, joining $1-c/2$ to $1-1/(2\log x)$; note that $\Gamma$ is entirely contained in the domain~$\cD$ defined in equation~\eqref{D and c def}.
\end{definition}

\begin{figure}[ht] \label{contour figure}
\includegraphics[height=4in]{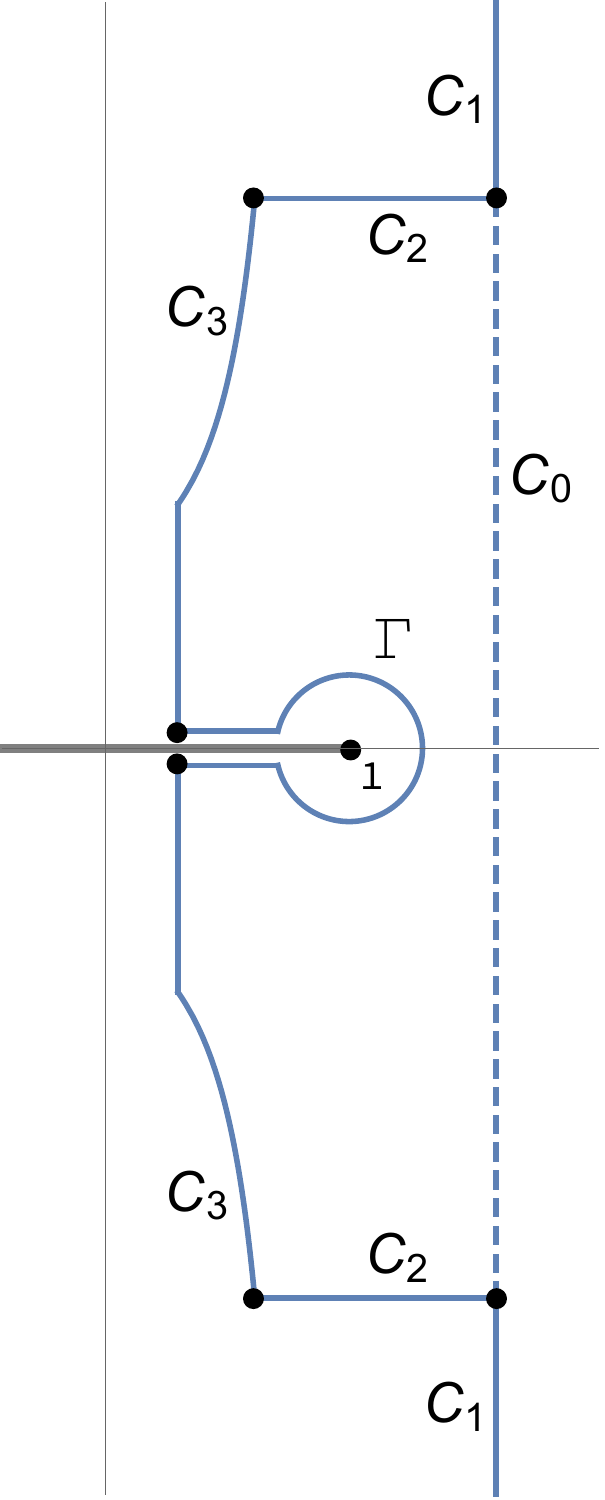}
\caption{The truncated Hankel contour $\Gamma$, and the contour of integration in Lemma~\ref{A integral lemma}}
\end{figure}

\begin{lemma} \label{Phi 2nd deriv lemma}
Suppose that $F(s)$ has the property $\P(z;c_0,\delta,M)$, where $|z|\le A$ and $c_0 \le \frac2{11}$. Then $\Phi''(x) \ll_A M(\log x)^A$ uniformly for $x>e^{48/(47c_0)}$.
\end{lemma}

\begin{proof}
Differentiating equation~\eqref{Phi def} under the integral sign yields
\begin{equation} \label{Phi derivatives}
\Phi'(x) = \frac{1}{2\pi i} \int_\Gamma F(s) x^s \,\frac{ds}{s} \quad\text{and}\quad \Phi''(x) = \frac{1}{2\pi i} \int_\Gamma F(s) x^{s-1} \,ds
\end{equation}
for $x>0$. The length of the contour $\Gamma$ is $O(1)$, and for $s\in\Gamma$ we have $|x^{s-1}|\le e^{1/2}$ and $F(s) \ll M |s-1|^A$ by Lemma~\ref{F bound lemma}, so that
\[
\Phi''(x) \ll |F(s)| \ll_{A} M|s-1|^{-A} \ll_A M(\log x)^A
\]
since $|s-1| \le 2\log x$.
\end{proof}

\begin{lemma} \label{Phi difference lemma}
Suppose that $F(s)$ has the property $\P(z;c_0,\delta,M)$, where $|z|\le A$ and $c_0 \le \frac2{11}$. Then
uniformly for $x\ge e^{2/c_0}$ and $|h| \le \frac x2$,
\[
\Phi(x+h) - \Phi(x) = h\Phi'(x) + O_A(Mh^2 (\log x)^A).
\]
\end{lemma}

\begin{proof}
The assertion follows immediately from the integration-by-parts formula
$
\Phi(x+h) - \Phi(x) = h\Phi'(x) + h^2 \int_0^1 (1-t) \Phi''(x+ht)\,dt
$
and Lemma~\ref{Phi 2nd deriv lemma} (noting that the assumption $x\ge e^{2/c_0}$ implies that $x+h \ge \frac x2 > e^{48/(47c_0)}$ since $c_0 \le \frac2{11}$).
\end{proof}

\subsection{Derivation of the Selberg--Delange formula}

Throughout this section, we write $a_n$ for the coefficients of the Dirichlet series~$F(s) = \sum_{n=1}^\infty a_nn^{-s}$, and we define the summatory function $A(t) = \sum_{n\le t} a_n$. The next three lemmas contain the majority of the work in establishing our version of the Selberg--Delange formula at the end of this section.

\begin{lemma} \label{A integral lemma}
Suppose that $F(s)$ has the property $\P(z;c_0,\delta,M)$, where $|z|\le A$ and $c_0 \le \frac2{11}$.
Then for $x>e^{1/c}$,
\[
\int_0^x A(t) \, dt = \Phi(x) + O_A\big( (\delta^{-A-1}+c_0^{-A}) Mx^2 e^{-\frac12\sqrt{c\delta\log x}} \big).
\]
\end{lemma}

\begin{proof}
By Perron's formula, we know that
\[
\int_0^x A(t) \, dt = \int_{(1+1/\log x)} F(s) x^{s+1} \frac{ds}{s(s+1)},
\]
where the contour integral is over the vertical line $\Re s = 1+1/\log x$ (the union of the segment $C_0$ and the two rays $C_1$ in Figure~$1$). 
Define $\sigma(t) = 1 - \frac c2/(1+\log^+ t)$.
By Cauchy's theorem, for any $T\ge3$ we may modify the path of integration into the union of:
\begin{itemize}
\item the vertical ray extending upward from $1+1/\log x+iT$ and its reflection in the real axis (labeled $C_1$ in Figure~$1$);
\item the horizontal segment from $\sigma(T)+iT$ to $1+1/\log x+iT$ and its reflection (labeled $C_2$);
\item the parametric curve $\sigma(\tau)+i\tau$ for $0^+\le\tau\le T$ and its reflection (labeled $C_3$);
\item the truncated Hankel contour $\Gamma$ from Definition~\ref{Gamma def}.
\end{itemize}
Note that this modified contour is entirely contained within the domain~$\cD$ defined in equation~\eqref{D and c def}.

When the first case of Lemma~\ref{F bound lemma}, which makes explicit the dependence of the upper bound on~$\delta$, is used in Tenenbaum's argument in place of~\cite[Part~II, equation~(5.19)]{tenenbaum15}, the resulting estimate for the contributions of $C_1$ and $C_2$ to the modified contour integral is $O_A(\delta^{-A}Mx^2T^{-\delta/2})$.

Tenenbaum's argument for the contribution of $C_3$ to the modified contour integral may be used for $|\tau|\ge3$, yielding an upper bound of $O_A(\delta^{-A-1}Mx^{1+\sigma(T)})$. However, to make explicit the dependence upon $c$ and hence $c_0$, we must use the second case of Lemma~\ref{F bound lemma} for the portion of $C_3$ with $|\tau|\le3$: we obtain
\begin{align*}
\ll_A \int M|s-1|^{-A} |x^{s+1}| \frac{|ds|}{|s^2|} &\ll M x^{2-c/5} \int_0^3 \frac{(c/2+\tau)^{-A}}{(1+\tau)^2} \,d\tau \\
&\ll_A M x^{2-c/5} \bigg( \int_0^c \frac{c^{-A}}{(1+\tau)^2} \,d\tau + \int_c^3 \frac{(1+\tau)^{-A}}{(1+\tau)^2} \,d\tau \bigg) \\
&\ll_A M x^{2-c/5} c^{-A} \ll_A M x^{2-c/5} c_0^{-A}.
\end{align*}
Summarizing the proof thus far, we have shown that
\begin{align}
\int_0^x A(t) \, dt &= \int_\Gamma \frac{F(s) x^{s+1}\, ds}{s(s+1)} + \int_{C_1\cup C_2} \frac{F(s) x^{s+1}\, ds}{s(s+1)} + \int_{C_3} \frac{F(s) x^{s+1}\, ds}{s(s+1)} \notag \\
&= \Phi(x) + O_A\big( \delta^{-A}Mx^2T^{-\delta/2}  + ( \delta^{-A-1}Mx^{1+\sigma(T)} + M x^{2-c/5} c_0^{-A} ) \big).
\label{our sigma plug in}
\end{align}
As Tenenbaum does, we choose $T = e^{\sqrt{(c/\delta)\log x}}$, so that 
\[
\sigma(T) = 1 - \frac{c/2}{1+\sqrt{(c/\delta)\log x}} \le \begin{cases}
1, &\text{if } x < e^{7\delta/c}, \\
1-\frac12\sqrt{\frac{c\delta}{\log x}} + \frac\delta{2\log x}, &\text{if } x \ge e^{7\delta/c}
\end{cases}
\]
(where the second case of the inequality is an easy calculus exercise).
In the case $x \ge e^{7\delta/c}$, the error term in equation~\eqref{our sigma plug in} becomes
\begin{equation} \label{expression}
O_A( \delta^{-A}Mx^2 e^{-\frac12\sqrt{c\delta\log x}} + \delta^{-A-1}Mx^2 e^{-\frac12\sqrt{c\delta\log x}} e^{\delta/2} + M x^{2-c/5} c_0^{-A} ).
\end{equation}
Since $\delta\le1$ we have $e^{\delta/2}\ll1$, and we also have $x^{-c/5} \le e^{-\frac12\sqrt{c\delta\log x}}$ for $x \ge e^{7\delta/c}$; thus the above expression simplifies to
\[
O_A\big( (\delta^{-A-1}+c_0^{-A}) Mx^2 e^{-\frac12\sqrt{c\delta\log x}} \big).
\]
But this bound for the expression~\eqref{expression} also holds for $x < e^{7\delta/c}$, as $x^{-c/5} < 1 \ll e^{-\frac12\sqrt{c\delta\log x}}$ in this range.
\end{proof}

We continue to follow the structure of Tenenbaum's proof; however, in the next proposition we give a somewhat different form for the error term, to simplify the dependence upon~$c_0$.

\begin{lemma} \label{Phi' asymptotic lemma}
Suppose that $F(s)$ has the property $\P(z;c_0,\delta,M)$, where $|z|\le A$ and $c_0 \le \frac2{11}$.
For $x>e^{48/(47c_0)}$ and for any $N\in\N$,
\[
\Phi'(x) = x(\log x)^{z-1} \bigg( \sum_{k=0}^N \frac{\lambda_k(z)}{(\log x)^k} + O_A \bigg( M\Gamma(N+A+2)\bigg( \frac{48}{c_0\log x} \bigg)^{\!\!N+1\,} \bigg) \bigg),
\]
where the functions~$\lambda_k(z)$ are defined in equation~\eqref{lambda_k def} below.
\end{lemma}

\begin{proof}
Since $F(s) = G(s;z) \zeta(s)^z$ by assumption, Definition~\ref{Z def} and equation~\eqref{Phi derivatives} imply that
\[
\Phi'(x) = \frac1{2\pi i} \int_\Gamma \frac{G(s;z)Z(s;z)x^s}{(s-1)^z} \,ds.
\]
Lemma~\ref{Z bound lemma} and the last remark of Definition~\ref{property P def} imply that $G(s;z)Z(s;z) \ll_A M$ uniformly for $|s-1|\le\frac2{11}$ and thus for $|s-1| \le c$.
If we define $\mu_k(z) = \Gamma(z-k)\lambda_k(z)$ (compare to~\cite[Part~II, equation~(5.24)]{tenenbaum15} and the preceding equation),
it follows immediately that $\mu_k(z) \ll_A Mc^{-k}$ for all $k\ge0$.
The derivation of~\cite[Part~II, equation~(5.25)]{tenenbaum15} and the following display proceeds with no changes other than clarifying the dependence on the parameters, yielding
\begin{equation} \label{first Phi' expansion}
\Phi'(x) = \sum_{k=0}^N \mu_k(z) \frac1{2\pi i} \int_\Gamma x^s(s-1)^{k-z} \,ds + O_A\big( Mc^{-N} x(\log x)^{\Re z-N-2} \Gamma(N+A+2) \big).
\end{equation}
The subsequent change of variables $w=(s-1)\log x$ leads, using~\cite[Part~II, Corollary~0.18]{tenenbaum15},~to
\begin{align*}
\frac1{2\pi i} \int_\Gamma x^s(s-1)^{k-z} \,ds &= \frac x{2\pi i}(\log x)^{z-1-k} \int_{\mathcal H(\frac12c\log x)} w^{k-z}e^w\,dw \\
&= x(\log x)^{z-1-k} \bigg( \frac1{\Gamma(z-k)} + O\big( 47^{|z-k|} \Gamma(1+|z-k|) e^{-\frac14c\log x} \big) \bigg).
\end{align*}
Using this identity in equation~\eqref{first Phi' expansion} gives
\begin{align} \label{second Phi' expansion}
\Phi'(x) &= \sum_{k=0}^N \mu_k(z) x(\log x)^{z-1-k} \bigg( \frac1{\Gamma(z-k)} + O\big( x^{-c/4} 47^{|z-k|} \Gamma(1+|z-k|) \big) \bigg) \notag \\
&\qquad{}+ O_A\big( Mc^{-N} x(\log x)^{\Re z-N-2} \Gamma(N+A+2) \big) \notag \\
&= x(\log x)^{z-1} \bigg( \sum_{k=0}^N \frac{\lambda_k(z)}{(\log x)^k}  + O\bigg( x^{-c/4} \sum_{k=0}^N \frac{|\mu_k(z)|}{(\log x)^k} 47^{|z-k|} \Gamma(1+|z-k|) \bigg) \bigg) \notag \\
&\qquad{}+ O_A\big( Mc^{-N} x(\log x)^{\Re z-N-2} \Gamma(N+A+2) \big) \notag \\
&= x(\log x)^{z-1} \bigg\{ \sum_{k=0}^N \frac{\lambda_k(z)}{(\log x)^k}  + O_A\bigg( Mx^{-c/4} \sum_{k=0}^N \bigg( \frac{47}{c\log x} \bigg)^{\!\!k\,} \Gamma(k+A+1) \bigg) \notag \\
&\qquad{}+ O_A\big( Mc^{-N} (\log x)^{-N-1} \Gamma(N+A+2) \big) \bigg\}.
\end{align}
Since $\Gamma(N+A+1) = (N+A)(N+A-1)\cdots(k+A+1)\Gamma(k+A+1)$, the sum in the error term above can be estimated (changing indices $j=N-k$) by
\begin{align*}
\sum_{k=0}^N \bigg( \frac{47}{c\log x} \bigg)^{\!\!k\,} \Gamma(k+A+1) &= \Gamma(N+A+1) \bigg( \frac{47}{c\log x} \bigg)^{\!\!N\,} \sum_{j=0}^N \bigg( \frac{c\log x}{47} \bigg)^{\!\!j\,} \prod_{i=1}^{j} \frac1{N-j+A+i} \\
&\le \Gamma(N+A+1) \bigg( \frac{47}{c\log x} \bigg)^{\!\!N\,} \sum_{j=0}^\infty \bigg( \frac{c\log x}{47} \bigg)^{\!\!j\,} \prod_{i=1}^{j} \frac1{i} \\
&= \Gamma(N+A+1) \bigg( \frac{47}{c\log x} \bigg)^{\!\!N\,} e^{(c\log x)/47} \\
&< \Gamma(N+A+1) \bigg( \frac{47}{c\log x} \bigg)^{\!\!N+1\,} x^{c/4},
\end{align*}
where the last inequality is equivalent to
$x^{c/47} < 47 x^{c/4} / (c\log x)$
which can be verified for $x>e^{48/(47c_0)}=e^{1/c}$ by an easy calculus exercise.
Therefore equation~\eqref{second Phi' expansion} becomes
\begin{align*}
\Phi'(x) &= x(\log x)^{z-1} \bigg\{ \sum_{k=0}^N \frac{\lambda_k(z)}{(\log x)^k} + O_A\bigg( M\Gamma(N+A+1) \bigg( \frac{47}{c\log x} \bigg)^{\!\!N+1\,} \bigg) \\
&\qquad{}+ O_A\big( Mc^{-N} (\log x)^{-N-1} \Gamma(N+A+2) \big) \bigg\}, 
\end{align*}
which implies the statement of the lemma.
\end{proof}

\begin{lemma} \label{A difference integral lemma}
Suppose that the Dirichlet series $F(s) = \sum_{n=1}^\infty a_n n^{-s}$ has type $\T(z,w;c_0,\delta,M)$, where $|z|\le A$ and $|w|\le A$ and $c_0 \le \frac2{11}$. For any $x\ge e^{2/c_0}$ and $0\le h\le \frac x2$,
\[
\int_x^{x+h} |A(t)-A(x)| \,dt \ll_A Mh^2 (\log x)^A + (\delta^{-A-1}+c_0^{-A}) Mx^2 e^{-\frac12\sqrt{c\delta\log x}}.
\]
\end{lemma}

\begin{proof}
By assumption, there exist nonnegative real numbers $b_n$, with $|a_n| \le b_n$ for all $n\ge1$, such that the Dirichlet series $\sum_{n=1}^\infty b_n n^{-s}$ has the property $\P(w;c_0,\delta,M)$. If we let $B(t) = \sum_{n\le t} b_n$, which is an increasing function, then the triangle inequality implies that $|A(t)-A(x)| \le B(t) - B(x)$ for all $t\ge x$, and thus
\begin{equation} \label{A to B ineq}
\int_x^{x+h} |A(t)-A(x)| \,dt \le \int_x^{x+h} (B(t)-B(x)) \,dt \le \int_x^{x+h} B(t) \,dt - \int_{x-h}^x B(t) \,dt.
\end{equation}
By Lemmas~\ref{Phi difference lemma} and~\ref{A integral lemma}, applied to $G(s) = \sum_{n=1}^\infty b_nn^{-s}$ in place of $F(s)$ and $B(t)$ in place of $A(t)$ and $w$ in place of~$z$, there exists a function $\Phi_1$ such that
\[
\Phi_1(x\pm h) - \Phi_1(x) = \pm h\Phi'_1(x) + O_A(Mh^2 (\log x)^A)
\]
and
\[
\int_0^x B(t) \, dt = \Phi_1(x) + O_A\big( (\delta^{-A-1}+c_0^{-A}) Mx^2 e^{-\frac12\sqrt{c\delta\log x}} \big),
\]
which implies that
\begin{align*}
\int_x^{x+h} B(t) \,dt - \int_{x-h}^x B(t) \,dt &= \big( \Phi_1(x+h) - \Phi_1(x) \big) + \big( \Phi_1(x-h) - \Phi_1(x) \big) \\
&\qquad{}+ O_A\big( (\delta^{-A-1}+c_0^{-A}) Mx^2 e^{-\frac12\sqrt{c\delta\log x}} \big) \\
&= h\Phi'_1(x) + (-h)\Phi'_1(x) \\
&\qquad{}+ O_A\big( Mh^2 (\log x)^A + (\delta^{-A-1}+c_0^{-A}) Mx^2 e^{-\frac12\sqrt{c\delta\log x}} \big).
\end{align*}
Combining this estimate with equation~\eqref{A to B ineq} establishes the lemma.
\end{proof}

With these tools in hand, we are ready to establish our version of the Selberg--Delange formula, which can be compared with~\cite[Part~II, Theorem~5.2]{tenenbaum15}.

\begin{theorem} \label{thm_SD}
Let $A\ge1$ be a real number and $N$ a nonnegative integer. Let $z$ and $w$ be complex numbers with $|z|\le A$ and $|w|\le A$, and let $c_0$, $\delta$, and $M$ be positive real numbers with $c_0\le\frac2{11}$ and $\delta\le1$. Let $F(s) = \sum_{n=1}^{\infty} a_n n^{-s}$ be a Dirichlet series that has type $\T(z,w;c_0,\delta,M)$ as in Definition~\ref{type T def}. Then, uniformly for $x\ge \exp\big( 8^{1/A} \max\{ \delta^{-1-1/A}, 2/c_0 \} \big)$,
$$
\sum_{n \le x} a_n = x(\log x)^{z-1} \bigg( \sum_{k=0}^N \frac{\lambda_k(z)}{(\log x)^k} + O_{A}(MR_N(x)) \bigg),
$$
where
\begin{equation} \label{lambda_k def}
\lambda_k(z) = \frac{1}{\Gamma(z-k)} \sum_{h=0}^k \frac{1}{h!(k-h)!}\frac{\partial^hG}{\partial s^h} (1;z)\gamma_{k-h}(z)
\end{equation}
with $G(s;z) = F(s) \zeta(s)^{-z}$ and with $\gamma_j(z)$ as in Definition~\ref{Z def},
and where
$$
R_N(x) = (\delta^{-2A-3/2}+c_0^{-2A-1})Mx e^{-\frac16\sqrt{c_0\delta\log x}} + M\Gamma(N+A+2)\bigg( \frac{48}{c_0\log x} \bigg)^{\!\!N+1\,} .
$$
\end{theorem}

\begin{proof}
Note that the assumed lower bound on~$x$ implies that $x \ge e^{2/c_0} > e^{1/c}$.
Using Lemma~\ref{A integral lemma}, for any $0\le h\le\frac x2$ we have
\begin{align*}
A(x) &= \frac1h \int_x^{x+h} A(t)\,dt - \frac1h \int_x^{x+h} (A(t)-A(x))\,dt \\
&= \frac1h \big( \Phi(x+h) - \Phi(x) + O_A\big( (\delta^{-A-1}+c_0^{-A}) Mx^2 e^{-\frac12\sqrt{c\delta\log x}} \big) \big) \\
&\qquad{}+ O\bigg( \frac1h \int_x^{x+h} |A(t)-A(x)| \,dt \bigg).
\end{align*}
Applying Lemmas~\ref{Phi difference lemma} and~\ref{A difference integral lemma} converts this expression into
\begin{align*}
A(x) &= \Phi'(x) + O_A\big( Mh (\log x)^A + (\delta^{-A-1}+c_0^{-A}) Mx^2 h^{-1} e^{-\frac12\sqrt{c\delta\log x}} \big).
\end{align*}

With the assumed lower bound on~$x$, each of the quantities $\delta^{-A-1}(\log x)^{-A}$ and $c_0^{-A}(\log x)^{-A}$ is at most $\frac18$, which implies that $(\delta^{-A-1}+c_0^{-A})(\log x)^{-A} \le \frac14$. In particular, we may take $h = (\delta^{-A-1}+c_0^{-A})^{1/2} x(\log x)^{-A/2}e^{-\frac14\sqrt{c\delta\log x}}$, which is less than~$\frac x2$; we obtain
\begin{align*}
A(x) &= \Phi'(x) + O_A\big( (\delta^{-A-1}+c_0^{-A})^{1/2} Mx (\log x)^{A/2} e^{-\frac14\sqrt{c\delta\log x}} \big) \\
&= x(\log x)^{z-1} \bigg\{ \sum_{k=0}^N \frac{\lambda_k(z)}{(\log x)^k} + O_A \bigg( M\Gamma(N+A+2)\bigg( \frac{48}{c_0\log x} \bigg)^{\!\!N+1\,} \bigg) \\
&\qquad{}+ O_A\big( (\delta^{-A-1}+c_0^{-A})^{1/2} M (\log x)^{3A/2+1} e^{-\frac14\sqrt{c\delta\log x}} \big) \bigg\} \\
&= x(\log x)^{z-1} \bigg\{ \sum_{k=0}^N \frac{\lambda_k(z)}{(\log x)^k} + O_A \bigg( M\Gamma(N+A+2)\bigg( \frac{48}{c_0\log x} \bigg)^{\!\!N+1\,} \bigg) \\
&\qquad{}+ O_A\big( (\delta^{-A-1}+c_0^{-A})^{1/2} M (c\delta)^{-3A/2-1} e^{-\frac15\sqrt{c\delta\log x}} \big) \bigg\},
\end{align*}
where the middle equality is from Lemma~\ref{Phi' asymptotic lemma} and the last equality is a simple calculus exercise. The theorem now follows from the fact that $\frac15\sqrt c=\frac15\sqrt{\frac{47}{48}c_0} > \frac16\sqrt{c_0}$.
\end{proof}

\begin{corollary} \label{SD corollary}
Let $z$ be a complex number with $|z|\le 1$, and let $c_0$ and $M$ be positive real numbers with $c_0\le\frac2{11}$. Let $F(s) = \sum_{n=1}^{\infty} a_n n^{-s}$ be a Dirichlet series with nonnegative coefficients that has the property $\P(z;c_0,\frac12,M)$ as in Definition~\ref{property P def}. Then, uniformly for $x\ge e^{16/c_0}$,
$$
\sum_{n \le x} a_n = x(\log x)^{z-1} \big( \lambda_0(z) + O(MR_0(x)) \big),
$$
where
$$
\lambda_0(z) = \frac{G(1;z)}{\Gamma(z)} = \frac{1}{\Gamma(z)} \lim_{s\to1} F(s)\zeta(s)^{-z}
$$
and
$$
R_0(x) = c_0^{-3} \big( e^{-\frac19\sqrt{c_0\log x}} + (\log x)^{-1} \big).
$$
\end{corollary}

\begin{proof}
The fact that the~$a_n$ are nonnegative means that $F(s)$ automatically has type $\T(z,z;c_0,\frac12,M)$ as per Definition~\ref{type T def}. We take $\delta=\frac12$ and $A=1$ in Theorem~\ref{thm_SD}, applying it directly with $N=0$. By equation~\eqref{lambda_k def}, we have simply $\lambda_0(z) = G(1;z)\gamma_0(z)/\Gamma(z)$; but by Definition~\ref{Z def}, $\gamma_0(z) = G(1;z) = 1$ identically, since the residue of the simple pole of $\zeta(s)$ at $s=1$ equals~$1$.
\end{proof}

\section*{Acknowledgments}

The authors were supported in part by a Natural Sciences and Engineering Research Council of Canada Discovery Grant and an Undergraduate Student Research Award.

\bibliography{SIFMG}
\bibliographystyle{plain}

\end{document}